\documentclass{amsart}

\usepackage{graphicx}
\usepackage{amsfonts}
\usepackage{amsmath}
\usepackage{amsthm}
\usepackage[vcentermath,enableskew]{youngtab}

\newtheorem{Thm}{Theorem}
\newtheorem*{Thm*}{Theorem}

\newtheorem{Lem}{Lemma}
\newtheorem{Prop}{Proposition}

\newtheorem{Rem}{Remark}

\newcommand{\q}{\mathbb{Q}}
\newcommand{\qss}{\ensuremath{{\bf QI}_m}}
\newcommand{\QI}{\ensuremath{ Q_T^{~k,m}}}
\newcommand{\Q}[2]{\ensuremath{ Q_T^{~{#1},{#2}}}}
\newcommand{\ptl}[2]{\ensuremath{ \frac{\partial^{#1}}{\partial {#2}^{#1}}}}
\newcommand{\Qdown}{\ensuremath{ Q_T^{~n-i+k,m-1}}}

\newcommand{\ie} {{\it i.e.,\ }}

\newcommand{\sgn}{\mathrm{sgn}}

\newcommand{\limj}{\lim_{x_1 \to x_j}}

\newcommand{\ddtw}{{\partial \over \partial x_2}}
\newcommand{\ddtwm}{{\partial^m \over \partial x_2^m}}

\numberwithin{equation}{section}

\begin{document}
\title{A new characterization for the $m$-quasiinvariants \\ of $S_n$ and explicit basis for two row hook shapes}
\author{Jason Bandlow}
\author{Gregg Musiker}
\thanks{Research was supported in part by NSF grant DMS-0500557.}

\begin{abstract}
In 2002, Feigin and Veselov \cite{FV} defined the space of
$m$-quasiinvariants for any Coxeter group, building on earlier work of \cite{CV}.  While many properties of
those spaces were proven in \cite{EG, FV, FV2, GW2} from this definition, an explicit
computation of a basis was only done in certain cases.  In particular, in \cite{FV}, bases for 
$m$-quasiinvariants were computed for dihedral groups, including $S_3$, and Felder and Veselov \cite{FV2} 
also computed the non-symmetric $m$-quasiinvariants of lowest degree for general $S_n$.  
In this paper, we provide a new characterization of the $m$-quasiinvariants of
$S_n$, and use this to provide a basis for the isotypic component
indexed by the partition $[n-1,1]$.
This builds on a previous paper, \cite{QuasiI}, in which we
computed a basis for $S_3$ via combinatorial methods.

\end{abstract}
\maketitle

\tableofcontents

\section{Introduction}
A permutation $\sigma \in S_n$ acts on a polynomial in ${\bf R} =
\q[x_1,\dots,x_n]$ by permutation of indices: 
$$ \sigma P(x_1, \dots, x_n) = P(x_{\sigma(1)}, \dots, x_{\sigma(n)}). 
$$ 
The $S_n$-invariant polynomials are known as symmetric functions, and
denoted by $\Lambda_n$.  It is well known that $\Lambda_n$ is generated by
the elementary symmetric functions $\{e_1, \dots, e_n\}$ where $$ e_j
= \sum_{i_1 < i_2 < \dots < i_j}^{} x_{i_1} \dots x_{i_j}.  $$
The ring of coinvariants of $S_n$ is the quotient
$$ {\bf R} / \langle e_1, \dots, e_n \rangle. $$
As an $S_n$-module, the ring of coinvariants is known to be isomorphic
to the left regular representation.
It is also known that ${\bf R}$ is free over $\Lambda_n$ which implies
that if we choose a basis ${\mathcal B} = \{ b_1, \dots, b_{n!} \}$
for the ring of coinvariants, any element of $P \in {\bf R}$ has a
unique expansion
$$ P = \sum_{i = 1}^{n!} b_i f_i $$
where the $f_i$ are symmetric functions.  More information is
given by the Hilbert series for the isotypic component of ${\bf R}$
corresponding to $\lambda$, namely
$$ \frac{\sum_{T \in ST(\lambda)} f_\lambda~
q^{cocharge(T)}}{(1-q)(1-q^2)\dots(1-q^n)} .$$
Known bases for the ring of coinvariants with very combinatorial
descriptions include the Artin monomials and the descent monomials.

In \cite{CV,FV}, Chalykh, Feigin and Veselov introduced a generalization of invariance known as
``$m$-quasiinvariance''.  For the symmetric group the
$m$-quasiinvariants are the polynomials $P \in \q [x_1, \dots, x_n]$ which have the divisibility
property
 $$(x_i - x_j)^{2m+1} \biggm| \bigg( 1 - (i,j) \bigg) P$$ for every transposition $(i,j)$.
We set
$$\qss = \{ \textrm{$m$-quasiinvariants of $S_n$} \}.$$
The $m$-quasiinvariants of $S_n$ form a ring and an $S_n$ module, and
we have the following containments:
$$ {\bf R} = {\bf QI}_0 \supset {\bf QI}_1 \supset
\cdots \supset {\bf QI}_m \supset \cdots \supset \Lambda_n.$$
For all $m$, the ring of coinvariants $\qss / \langle e_1, \dots, e_n \rangle$
was conjectured in \cite{FV}, and proved in \cite{EG}, to be isomorphic as an $S_n$-module to the left regular
representation.  In fact, Etingof and Ginzburg further proved that $\qss$ is free over the symmetric functions.
The Hilbert series of the isotypic component indexed by $\lambda$
is given by \cite{FV2} to be  
\begin{align}
\frac{\sum_{T \in ST(\lambda)} f_\lambda~
q^{m\left( \binom{n}{2} -content(\lambda(T)) \right) +
cocharge(T)}}{(1-q)(1-q^2)\dots(1-q^n)} . \label{eq:Hilb}
\end{align}
Here $content$ and $cocharge$ are two statistics on tableaux--we will
not need the precise definitions.  In fact $content$ only depends on
the shape of $T$ hence it is actually a function on partitions.

In light of the simple combinatorial descriptions of a basis for the
coinvariants in the classical (or $m = 0$) case, the authors have
looked for a basis for larger $m$.  In \cite{QuasiI} and \cite{FV} a basis was given for
the case $n=3$.  (The work \cite{FV} specifically described the quasiinvariants for dihedral groups, so in particular for $D_3 \cong S_3$.)  
Further, in \cite{FV2}, Felder and Veselov provide integral expressions, $\phi^{(j)}(x)$ for $2 \leq j \leq n$,
for the lowest degree (non-symmetric) $m$-quasiinvariants, i.e. those 
of degree $mn+1$.  In the present work, we give a complete basis of the
isotypic component given by the partition $[n-1,1]$ for any $n$.  This is
accomplished by means of a new characterization of $\qss$:

\newtheorem*{ThmMain}{Theorem~\ref{Thm:Main}}
\begin{ThmMain}
The vector space of quasiinvariants has the following direct sum
decomposition:
\begin{align*}
  \qss = \bigoplus_{T \in ST(n)} \left( \gamma_T{\bf R} \cap V_T^{2m+1}{\bf R} \right)
\end{align*}
\end{ThmMain}
\noindent where $ST(n)$ is the set of standard tableaux of size $n$, $\gamma_T$
is a projection opertor due to Young (defined in full detail in the
next section) and $V_T$ is the polynomial given by the product over
the columns of $T$ of the associated ``Vandermonde determinants''
(this is also defined in detail below).  This characterization is
proved using completely elementary methods (namely, computations in
the group algebra of the symmetric group) in
section~\ref{Sec:NewChar}.  In section~\ref{lequals2} we use this
characterization to construct the basis for the
$[n-1,1]$ isotypic component.  Precisely, for $T$ a standard Young
tableau of shape $[n-1,1]$ with $j$ the entry in the second row, we
set
$$Q_T^{k,m} = \int_{x_1}^{x_j} t^k \prod_{i=1}^n (t-x_i)^m dt.$$
With this definition, we have
\newtheorem*{ThmBasis}{Theorem~\ref{Thm:Basis}}
\begin{ThmBasis}
The set 
$$\{Q_T^{0,m}, Q_T^{1,m}, Q_T^{2,m}, \dots, Q_T^{n-2,m}\}
$$ 
is a basis for $\gamma_T \left( \qss / \langle e_1, \dots, e_n \rangle
\right)$.
\end{ThmBasis}
In section~\ref{Sec:EvalIntegrals}  we evaluate the integrals that represent these polynomials in a more explicit form.

Along the journey to these results, the authors discovered other
interesting facts about the ring $\qss$.  In
section~\ref{Sec:LmAction}, we show that the operator $$L_m =
\sum_{i=1}^n \frac{\partial^2}{\partial x_i^2} - 2m \sum_{1 \leq i <
j \leq n} \frac{1}{x_i - x_j} \left( \ptl{}{x_i} - \ptl{}{x_j}
\right)$$
acts on our basis by the simple formula
$$ L_m Q_T^{k,m} = k(k-1) Q_T^{k-2,m}. $$  Finally, in
section~\ref{Sec:ChangeOfBasis} we show that if we think of
${\bf QI}_{m+1}$ and $\qss$ as modules over the ring $\Lambda_n$, the
determinant of the respective change of basis matrix is the
Vandermonde determinant to the power $n!$, regardless of the value of
$m$. We hope that these results prove as suggestive to others as to
ourselves, and spur further investigations into this newly discovered
territory.

\section{Definitions and Notation} \label{Sec:DefAndNot}

Throughout this paper, we will write elements of the symmetric group
$S_n$ using cycle notation.  We will perform many calculations in the
group algebra of $S_n$, and as such it will be useful to have
shorthand notation for many commonly occurring elements. For a given
subgroup $A$ of $S_n$, we set
\begin{align*}
  [A] &= \sum_{\sigma \in A} \sigma \qquad \text{ and }\\
  [A]' &= \sum_{\sigma \in A} \sgn(\sigma)\sigma.
\end{align*}
We will extend this notation, abusing it slightly, and also define,
for any set $U$ whatsoever,
\begin{align*}
  [U] &= \sum_{\sigma \in S_U} \sigma \qquad \text{ and }\\
  [U]' &= \sum_{\sigma \in S_U} \sgn(\sigma)\sigma.
\end{align*}

The Young diagram of a partition $\lambda$ is a subset of the boxes in
the positive integer lattice, indexed by ordered pairs $(i,j)$, where
$i$ is the row index and $j$ is the column index.  For example, in the
following Young diagram of $[4,3,2]$, the cell $(2,3)$ is marked:
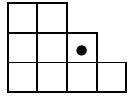
\[
\young(\hfil \hfil,\hfil \hfil \bullet,\hfil \hfil \hfil \hfil)
\]

A \emph{tableau} of shape $\lambda \vdash n$ is a function from the
cells of the Young diagram of $\lambda$ to the set $\{1, \dots,
n\}$.  We write the $T(i,j)$ for the value of $T$ at the cell $(i,j)$.
For example, if $T$ is the following tableau, $T(2,3) = 8$:
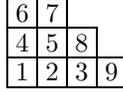
\[
\young(67,458,1239)
\]

We call a tableau \emph{standard} if it is injective and the entries
increase across the rows and up the columns.  For example, the tableau
above is standard.  We denote the set of standard tableaux of shape
$\lambda$ by $ST(\lambda)$ and the set of all standard tableaux with
$n$ boxes by $ST(n)$.

Given a tableau $T$ we let $C_i$ be the set of elements in the
$i^{th}$ column and we define $R_i$ similarly for the rows.  We also
set
\begin{align*}
C(T) &= \{ (i,j) \in S_n \mid i,j \textrm{ are in the same column of }
T \}\\
R(T) &= \{ (i,j) \in S_n \mid i,j \textrm{ are in the same row of } T
\} \\
N(T) &= \prod_i [C_i]' \\
P(T) &= \prod_i [R_i] \\
f_{\lambda} &= \textrm{ the number of standard tableaux of shape $\lambda$} \\
\gamma_T &= \frac{f_\lambda~N(T) P(T)}{n!} \\
\lambda(T) &=\textrm{ the shape of tableau } T. 
\end{align*}
Finally, we define the following useful polynomial associated with a
tableau $T$:
\begin{align*}
V_T &= \prod_{(i,j) \in C(T)} (x_i - x_j).
\end{align*}

\section{Useful Facts About $\q S_n$ modules} \label{Sec:QSnMods}

The fundamental theorem of representation theory states
\begin{Prop} For $W$ a finite dimensional $S_n$-module,
\begin{align*}
W \cong \bigoplus_{\lambda \vdash n} V_{\lambda}^{\oplus m_\lambda}
\end{align*}
where the $V_\lambda$ are the irreducible representations of $S_n$ and
the $m_\lambda$ are non-negative integers.
\end{Prop}
The vector space and $S_n$-module $V_\lambda^{\oplus m_\lambda}$ is
known as the \emph{isotypic component} of $V$ indexed by $\lambda$.
Now, \qss is infinite dimensional, but it is the direct sum of
homogeneous components, each of which are finite dimensional.  So we
have that each homogeneous component of \qss decomposes into the
direct sum of irreducibles.  The direct sum of all copies of
$V_\lambda$ occuring in this decomposition is still itself an 
$S_n$-module, and is still referred to as the isotypic component indexed by
$\lambda$.  However, we will find the following decomposition of $V$
more useful.
\begin{Prop}
  On any $S_n$ module $W$, the group algebra elements $\{\gamma_T\}_{T
  \in ST(n)}$ act as projection operators.  In symbols, we have the conditions
\begin{enumerate}
    \item{$\gamma_T^2 = \gamma_T$}
    \item{$W = \bigoplus_{T \in ST(n)} \gamma_T W$}.
\end{enumerate}
\end{Prop}

Note that in this decomposition, unlike the previous one, the direct
summands are not themselves $S_n$-modules.  We do have the
following proposition, however, nicely relating the previous two.

\begin{Prop}
  For any $S_n$ module $W$,
  $$\bigoplus_{T \in ST(\lambda)} \gamma_T W$$
  is the isotypic component of $W$ indexed by $\lambda$.
\end{Prop}

In the case of the quasiinvariants, we have the following
\begin{Prop} \label{Prop:directsum} 
  The $\q$-vector space of $m$-quasiinvariants has the following
  direct sum decomposition:
  \begin{align*}
    \qss = \bigoplus_{T \in ST(n)} \gamma_T \qss.
  \end{align*}
\end{Prop}

Our goal will be to use the decomposition $\qss / \langle e_1, \dots
e_n \rangle = \bigoplus_T \gamma_T \left( \qss / \langle e_1, \dots
e_n \rangle \right)$ to find a basis for this quotient module.

\section{A New Characterization of $S_n$-Quasiinvariants}
\label{Sec:NewChar}

In this section we prove the following theorem:
\begin{Thm} \label{Thm:Main} 
The vector space of quasiinvariants has the following direct sum
decomposition:
\begin{align*}
  \qss = \bigoplus_{T \in ST(n)} \left( \gamma_T{\bf R} \cap
  V_T^{2m+1}{\bf R} \right).
\end{align*}
\end{Thm}

We will prove this by showing
\begin{align}
  \gamma_T \qss = \gamma_T{\bf R} \cap V_T^{2m+1}{\bf R}.
  \label{eq:gammaTIntersection}
\end{align}
Combining (\ref{eq:gammaTIntersection}) with
Proposition~\ref{Prop:directsum} will prove the theorem.
Equation~(\ref{eq:gammaTIntersection}) is proved by considering some relations in the
group algebra of $S_n$.  We begin with the following simple
proposition:
\begin{Prop} \label{fpq}
Let $f = \sum_{\sigma \in S_n} f_\sigma \sigma \in \mathbb{Q} S_n$,
and $P,Q \in \mathbb{Q} [x_1, \dots x_n]$ with $P$ a symmetric
function. Then we have $f(PQ) = P f(Q)$.
\end{Prop}
\begin{proof}
We have the following calculation:
\begin{align*}
f(PQ) &= (\sum_{\sigma \in S_n} f_\sigma \sigma) (PQ) \\
&= \sum_{\sigma \in S_n} f_\sigma (\sigma P) (\sigma Q) \\
&= P \sum_{\sigma \in S_n} f_\sigma (\sigma Q) \\
&= P f(Q). \qedhere
\end{align*}
\end{proof}

\begin{Lem} \label{Sndecomp}
The group algebra element $[S_n]$ can be written as
\begin{align*}
  \bigg( 1+(i_1,i_2) \bigg) \bigg( 1 + (i_1,i_3) + (i_2,i_3) \bigg)
  \cdots \bigg( 1+(i_1,i_n)+(i_2,i_n)+\dots+(i_{n-1},i_n) \bigg)
\end{align*}
where $\{i_1, \dots,i_n\}$ is any permutation of $\{1,\dots,n\}$.
Similarly, $[S_n]^\prime$ can be written as
\begin{align*}
  \bigg( 1 - (i_1,i_2) \bigg) \bigg( 1 - (i_1,i_3) - (i_2,i_3) \bigg)
  \cdots \bigg( 1 - (i_1,i_n) - (i_2,i_n) - \dots - (i_{n-1},i_n)
  \bigg).
\end{align*}
\end{Lem}

\begin{proof} The statement is trivial for $n=1$.  Now assume the
  statement is true for $S_{n-1}$. Let $H$ be the subgroup of $S_n$
  consisting of all permutations which leave $i_n$ fixed. Right coset
  decomposition gives
\begin{align*}
  S_n = H + H(i_1,i_n) + H(i_2,i_n) + \dots + H(i_{n-1},i_n).
\end{align*}
Thus
\begin{align*}
[S_n] &= [H] \bigg( 1 + (i_1,i_n) + (i_2,i_n) + \dots + (i_{n-1},i_n)
\bigg) \text{ and }\\
[S_n]' &= [H]' \bigg( 1 - (i_1,i_n) - (i_2,i_n) - \dots -
(i_{n-1},i_n) \bigg).
\end{align*}
As $H$ is isomorphic to $S_{n-1}$ the statement is proved.
\end{proof}

\begin{Rem} Note that left coset decomposition could just as easily
have been used in this proof, which would give the factors in the
opposite order.
\end{Rem}

For the following, we fix the following:
\begin{itemize}
  \item  $T$ a tableau of shape $\lambda \vdash n$,
  \item  $i,j$ with $1 \le i < j \le \lambda_1$.
\end{itemize}
With $T$ fixed, we use the boldface notation ${\bf a_b}$ as shorthand
for $T(a,b)$, the element in the $a^{th}$ row and $b^{th}$ column of
$T$. In the following, we will make much use of elements of
$\q[S_n]$ of the form $[C_i \cup \left\{ {\bf k_j} \right\}]'$; the
signed sum of all permutations of the elements of column $i$, and a
single element ${\bf k_j}$ in column $j$ to the right of $i$. We first
note that elements of this form kill $P(T)$:

\begin{Lem} \label{zeroprod}
For any $k \in \{1, \dots, |C_j|\}$ we have
\[
[C_i \cup \{ {\bf k_j} \}]' P(T) = 0.
\]
\end{Lem}

\begin{proof}
Since the rows consist of disjoint elements, all factors of the form
$[R_k]$ in $P(T)$ commute, and we have
\begin{align*}
  [C_i \cup \{ {\bf k_j} \}]' P(T)
  &= [C_i \cup \{ {\bf k_j} \}]' [R_k] \prod_{l \neq  k} [R_l] \\
  \textrm{ (by Lemma~\ref{Sndecomp})} &= [C_i \cup \{ {\bf k_j} \}]'
  \bigg( 1 + \big({\bf k_i}, {\bf k_j} \big) \bigg) (\textrm{other factors}) \\
  &= \bigg( [C_i \cup \{ {\bf k_j} \}]' - [C_i \cup \{ {\bf k_j} \}]' \bigg) (\textrm{other factors}) \\
  &= 0. \qedhere
\end{align*}
\end{proof}

Given a column $C_i$ and an element ${\bf k_j}$ in a column $C_j$ to the right of
$i$, we denote by $\alpha_{i,{\bf k_j}}$ the sum of all transpositions
consisting of ${\bf k_j}$ and an element of $C_i$, \ie
\begin{align*}
  \alpha_{i,{\bf k_j}} &=\sum_{t=1}^{|C_i|} \big({\bf t_i}, {\bf k_j} \big)
\end{align*}

An important property of this element $\alpha_{i,{\bf k_j}}$ is the
following:
\begin{Lem} \label{thm1crux}
  The element $\alpha_{i,{\bf k_j}}$ leaves $\gamma(T)$ invariant, \ie
  \begin{align*}
   \alpha_{i,{\bf k_j}} \gamma(T) = \gamma(T)
 \end{align*}
\end{Lem}

\begin{proof}
  It suffices to show that $(1 - \alpha_{i,{\bf k_j}}) N(T) P(T) = 0$.
  The first step is to write $N(T)$ as $[C_i]' \prod_{r \neq i}
  [C_r]'$. We begin by noting that
  \begin{align} \label{crux1}
    (1 - \alpha_{i,{\bf{k_j}}})N(T)P(T) = \bigg(\prod_{t \ne i,j} [C_t]' \bigg) (1 -
    \alpha_{i,{\bf{k_j}}})[C_i]' [C_j]' P(T)
  \end{align}
  since the elements of $C_t$, for $t \not \in \{i,j\}$ are disjoint
  from $C_i \cup \{ {\bf k_j} \}$.  By Lemma \ref{Sndecomp} we have
  \begin{align}
    (1 - \alpha_{i,{\bf{k_j}}})[C_i]' &=  \bigg( 1 - \sum_{r=1}^{|C_i|} ({\bf r_i},{\bf k_j}) \bigg)[C_i]'
    &= [C_i \cup \{ {\bf k_j} \}]' \label{eq:crux2}
  \end{align}
  so substituting (\ref{eq:crux2}) into (\ref{crux1}) and expanding $[C_j]'$ by Lemma~\ref{Sndecomp} gives
  \begin{align*}
    (1 - \alpha_{i,{\bf{k_j}}}) & N(T)P(T)\\
    =& \bigg(\prod_{t \ne i,j} [C_t]' \bigg) [C_i \cup \{ {\bf k_j} \}]' [C_j]' P(T)\\
    =& \bigg(\prod_{t \ne i,j} [C_t]' \bigg) [C_i \cup \{ {\bf k_j} \}]'
     \bigg( 1 - ({\bf 1_j},{\bf 2_j}) \bigg) \\
    & \cdots \bigg( 1 - ({\bf{1_j}},{\bf k_j}) -
     ({\bf{2_j}},{\bf k_j}) - \dots - \widehat{({\bf{k_j}},{\bf k_j})} - \dots -
     ({\bf{|C_j|_j}},{\bf k_j}) \bigg) P(T).
  \end{align*}
  Moving the factors which do not involve ${\bf k_j}$ to the left and
  rewriting gives
  \begin{align*}
    (1 - \alpha_{i,{\bf{k_j}}})&N(T)P(T)\\
    =& \bigg( \text{other factors} \bigg) \left([C_i \cup \{ {\bf
    k_j}\}]' \right) \\
    & \bigg( 1 - ({\bf 1_j},{\bf k_j}) - ({\bf 2_j},{\bf k_j}) - \dots -
    \widehat{({\bf k_j},{\bf k_j})} - \dots - ({\bf |C_j|_j},{\bf
    k_j}) \bigg) \left( P(T) \right) \\
    =& \bigg( \text{other factors} \bigg) \\
    & \bigg( [C_i \cup \{ {\bf k_j} \}]' P(T)  -
    \sum_{\substack{t=0 \\ t \neq k}}^{|C_j|}
    [C_i \cup \{ {\bf k_j} \}]' ({\bf t_j}, {\bf k_j}) P(T) \bigg)
  \end{align*}

  We now use the fact that  $[C_i \cup \{ {\bf k_j} \}]' ({\bf t_j}, {\bf k_j}) = ({\bf t_j}, {\bf k_j})
  [C_i \cup \{ {\bf t_j} \}]'$ to obtain
  \begin{align*}
    (1 - \alpha_{i,{\bf{k_j}}})N(T)P(T)
    =& \bigg( \text{other factors} \bigg) \\
    & \bigg( [C_i \cup \{ {\bf k_j} \}]' P(T)  -
    \sum_{\substack{t=0 \\ t \neq k}}^r ({\bf t_j}, {\bf k_j}) [C_i
    \cup \{ {\bf t_j} \}]' P(T) \bigg) \\
    =& 0
  \end{align*}
  where the last equality follows from Lemma \ref{zeroprod}.
\end{proof}

We now have the tools to prove the difficult containment of
Theorem~\ref{Thm:Main}.
\begin{Lem} \label{Lem:GammaRCapVTRinGammaTQI} 
  For all standard tableaux $T$ and all $m \ge 0$, we have the
  following containment of vector spaces:
  \begin{align*}
    \gamma_T{\bf R} \cap V_T^{2m+1}{\bf R} \subseteq \gamma_T \qss.
  \end{align*}
\end{Lem}

\begin{proof}
  Since $\gamma_T$ is an idempotent, it suffices to show that for any
  polynomial $P$ in the ideal $V_T^{2m+1}{\bf R}, ~\gamma_T P = P$
  implies that $P$ is $m$-quasiinvariant.

  Let $P$ be such that $ V_T^{2m+1} | P$ and $\gamma_T P = P$.
  We wish to show that $\bigg( 1 - (a,b) \bigg) P$ is divisible by
  $(x_a - x_b)^{2m+1}$ for all transpositions $(a,b)$. We first
  consider the case where $a$ and $b$ are in the same column of $T$.
  In this case we have
  $$(a,b)N(T) = -N(T)$$
  and so
  $$(a,b)P = (a,b) \gamma_T P = -\gamma_T P = -P.$$
  Thus
  $$\bigg( 1 - (a,b) \bigg) P = 2P \in V_T^{2m+1}{\bf R}$$
  which is divisible by the required factor.

  Now suppose without loss that $a = {\bf k_i}$ is to the left of $b$ in
  column $C_j$.  By Lemma~\ref{thm1crux} $P$, is preserved by
  $\alpha_{i,b}$:
  \begin{align} \label{eq:apisp}
    \alpha_{i,b}P =\alpha_{i,b} \gamma_T P =
    \gamma_T P = P.
  \end{align}
  Equation~(\ref{eq:apisp}) gives
  \begin{align}
    \bigg( 1 - (a,b) \bigg) P &= P - (a,b)P \\
    &= \alpha_{i,b}P - (a,b) P \\
    &= \sum_{\substack{t=1 \\ t \neq k}}^{|C_i|} ({\bf t_i},b) P. \label{LHSRHS}
  \end{align}
  Since $P \in V_T^{2m+1}{\bf R}$, for any $t \in \{1,\dots,|C_i|\}$ with
  $t \neq k$ we can rewrite $P$ as
  $$P =  (x_{\bf t_i}-x_a)^{2m+1} (\textrm{other factors}).$$
  Thus
  $$({\bf t_i}, b) P = (x_b - x_a)^{2m+1} (\textrm{other factors})$$
  and we have
  $$ (x_b - x_a)^{2m+1} \textrm{ divides } ({\bf t_i}, b) P \textrm{ for
  every } t \in \{1,\dots,|C_i|\} \textrm{ with } t \neq k.$$
  Hence $(x_b - x_a)^{2m+1}$ divides the right-hand side of
  equation~\ref{LHSRHS}, which completes the proof.
\end{proof}

The proof of Theorem~\ref{Thm:Main} now follows easily.

\begin{proof}[Proof of Theorem~\ref{Thm:Main}]
  Lemma~\ref{Lem:GammaRCapVTRinGammaTQI} gives one containment.  It
  remains to show that
  \begin{align*}
    \gamma_T \qss \subseteq \gamma_T{\bf R} \cap V_T^{2m+1}{\bf R}.
  \end{align*}
  In particular, we must show that for $Q \in \qss$ we have
  $$\gamma_T Q \in V_T^{2m+1}{\bf R}.$$
  Let $P = \gamma_T Q = N(T)Q'$. $P$ must be anti-symmetric with
  respect to all transpositions in $C(T)$ since it is in the image of
  $N(T)$.  Thus, for any $(a,b) \in C(T), \bigg( 1-(a,b) \bigg) P =
  2P$.  Hence $(x_a-x_b)^{2m+1}$ divides $2P$ (and also $P$) for all
  $(a,b) \in C(T)$.  This establishes
  equation~(\ref{eq:gammaTIntersection}) and hence the theorem.
\end{proof}

\section{A Basis For The Isotypic Component $\lambda = [n-1,1]$} \label{lequals2}

In this section, we refer to the quotient $\qss / \langle e_1,
\dots, e_n \rangle$ by the symbol $\qss^*$.  Our object here is to
describe a basis for $\gamma_T \qss^*$ when $T$ has a hook shape of
the form $[n-1,1]$.  Until otherwise specified, let $\lambda$ be the
partition $[n-1,1]$ and let $T$ be one of the $(n-1)$ standard
tableaux of shape $\lambda$. In fact $T$ is uniquely defined by the
lone entry of the second row. Suppose it's $j \in \{2,3,\dots, n\}$.
We define
$$Q_T^{k,m} = \int_{x_1}^{x_j} t^k \prod_{i=1}^n (t-x_i)^m dt.$$
Our goal will be to show that the polynomials
$\{Q_T^{k,m}\}_{k=0}^{n-2}$ are a set of representatives for a basis
of $\gamma_T \qss^*$.  Before we do this, we show that these
polynomials satisfy a remarkable recursion. In what follows, $e_i$
will denote the $i$th elementary symmetric function in the variables
$x_1, \dots, x_n$, with the convention that $e_0 = 1$.

We first state for reference a classical symmetric function
identity:
\begin{align}\label{eident}
\prod_{i=1}^n (t-x_i) = \sum_{i=0}^n (-1)^i e_i  t^{n-i}.
\end{align}
We now state our recursion.
\begin{Prop} \label{recursion}
For $m > 1$ we have the identity
$$\QI = \sum_{i=0}^n (-1)^i e_i ~\Qdown.$$
\end{Prop}

\begin{proof}
Unpacking the product in the definition of $\QI$ we get
\begin{align}\label{recur1} \QI = \int_{x_1}^{x_j} \bigg (
\prod_{i=1}^n (t-x_i)\bigg ) t^k \prod_{l=1}^n (t-x_l)^{m-1} dt,
\end{align}
and substituting (\ref{eident}) into (\ref{recur1}) and pulling out
the factors not involving $t$ gives
\begin{align*}
\QI &= \int_{x_1}^{x_j} \bigg(\sum_{i=0}^n (-1)^i e_i  t^{n-i}\bigg) t^k \prod_{l=1}^n (t-x_l)^{m-1} dt \\
&= \sum_{i=0}^n (-1)^i e_i \int_{x_1}^{x_j} t^{n-i+k} \prod_{l=1}^n (t-x_l)^{m-1} dt \\
&= \sum_{i=0}^n (-1)^i e_i ~\Qdown.
\end{align*}
\end{proof}

\begin{Rem}
The polynomials defined by $Q_T^{0,m}$ above, as $T$ runs over the $(n-1)$ possible standard Young 
tableaux of shape $[n-1,1]$, agree with the evaluations of Felder and Veselov's $\phi^{(j)}(x)$'s up to a scalar even though our definitions differ.  
This is a consequence of the fact that the lowest degree polynomials in $\gamma_T {\bf QI}_m$ comprise a one-dimensional space.
\end{Rem}

We now show that we have $Q_T^{k,m} \in \gamma_T \qss$. By Theorem
\ref{Thm:Main} it is enough to show that we have $\QI \in \gamma_T
{\bf R}$ and $V_T^{2m+1} \bigm| \QI$.

\begin{Prop} \label{eTing}
The polynomial $\QI$ is invariant under the action of the group
algebra element $\gamma_T$.
\end{Prop}

\begin{proof} We first show the statement is true in the case $m=0$, and then proceed by induction. From the definition of $\QI$ we have
\begin{align}
Q_T^{k,0} &= \int_{x_1}^{x_j} t^k \prod_{i=1}^n (t-x_i)^0 dt \\
&= \frac{x_j^{k+1} - x_1^{k+1}}{k+1}. \label{qtk0}
\end{align}
Thus $Q_T^{k,0}$ is invariant under the transposition $(a,b)$ for
$a,b \in \{2,\dots, \hat{j},\dots, n\}$. This immediately gives
\begin{align} \label{eTing1}
[S_{ \{2,\dots, \hat{j}, \dots, n\}}] Q_T^{k,0} = (n-2)! Q_T^{k,0}.
\end{align}
Now, $P(T) = [S_{ \{1,2,\dots, \hat{j},\dots, n\} }]$ and expanding
this according to Lemma~\ref{Sndecomp} yields
\begin{align} \label{eTing2}
P(T) = \bigg( 1 + (1,2) +  \cdots + \widehat{(1,j)} + \cdots (1,n)
\bigg) [S_{ \{2,\dots, \hat{j}, \dots, n\} }].
\end{align}
Using (\ref{qtk0}), (\ref{eTing1}) and (\ref{eTing2}) and performing
a simple calculation, we obtain
\begin{align}
P(T) Q_T^{k,0} &= [S_{ \{1,2,\dots, \hat{j},\dots, n\} }] \frac{x_j^{k+1} - x_1^{k+1}}{k+1} \\
&= \bigg( 1 + (1,2) + \cdots + \widehat{(1,j)} + \cdots (1,n) \bigg) [S_{ \{2,\dots, \hat{j}, \dots, n\} }] \frac{x_j^{k+1} - x_1^{k+1}}{k+1} \\
&= \frac{(n-2)!}{k+1} \bigg( 1 + (1,2) + \cdots + \widehat{(1,j)} + \cdots (1,n) \bigg) (x_j^{k+1} - x_1^{k+1}) \\
&= \frac{(n-2)!}{k+1}  \bigg( (n-1) ( x_j^{k+1} ) - x_1^{k+1} -
(x_2^{k+1} + \dots + \widehat{x_j^{k+1}} + \dots + x_n^{k+1} )
\bigg) \label{eTing3}.
\end{align}
Since $N(T) = \bigg( 1 - (1,j) \bigg)$ we can use (\ref{eTing3}) to
get
\begin{align}
N(T)P(T) Q_T^{k,0} &= \frac{(n-2)!}{k+1} \bigg( 1 - (1,j) \bigg) \bigg( (n-1) x_j^{k+1} - x_1^{k+1} - (x_2^{k+1} + \dots + \widehat{x_j^{k+1}} + \dots + x_n^{k+1} ) \bigg) \\
&= \frac{(n-2)!}{k+1} \bigg( (n-1) (x_j^{k+1} - x_1^{k+1}) - (x_1^{k+1} - x_j^{k+1}) - 0 \bigg) \\
&= \frac{n(n-2)!}{k+1} (x_j^{k+1} - x_1^{k+1}). \label{eTing4}
\end{align}
Finally, we use (\ref{eTing4}) and the fact $f_\lambda = (n-1) = {n!
\over n(n-2)!}$ to reach the desired conclusion
\begin{align*}
\gamma_T Q_T^{k,0} &= \frac{n!}{f_\lambda} N(T)P(T)Q_T^{k,0} \\
&= {(x_j^{k+1} - x_1^{k+1}) \over k+1} \\
&= Q_T^{k,0}.
\end{align*}
With this in hand, we use Proposition \ref{recursion} to write
\begin{align*}
\gamma_T \QI &= \sum_{i=0}^n \gamma_T~(-1)^i e_i ~\Qdown.
\end{align*}
Applying Proposition \ref{fpq} 
and induction gives
\begin{align*}
\gamma_T \QI &= \sum_{i=0}^n (-1)^i e_i \big( \gamma_T \Qdown \big) \\
&= \sum_{i=0}^n (-1)^i e_i \big( \Qdown \big) \\
&= \QI.
\end{align*}
\end{proof}

In order to complete our task of showing that $\QI \in \gamma_T
\qss$, we must show that $(x_j - x_1)^{2m+1} \bigm| \Q{k}{m}$. We do
so by proving the following stronger statement:

\begin{Prop}\label{xjkfactor}
For all $k$,
\begin{align*}
\limj \frac{\Q{k}{m}}{(x_j - x_1)^{2m + 1}} =  \frac{(-1)^m
m!^2}{(2m+1)!} x_j^k \prod_{\stackrel{i=2}{i \neq j}}^n (x_j -
x_i)^m.
\end{align*}
\end{Prop}

\begin{proof}
This proof will rely on Leibniz's integral formula, also known as
the technique of differentiation underneath the integral sign. We
state the rule here for the reader's convenience. For $f(x,y), u(x),
v(x)$ continuous functions we have
\begin{align} \label{underintegral}
\frac{d}{dx} \int_{u(x)}^{v(x)} f(x,y) dy = \bigg(f(x,v(x))\cdot
\frac{\partial v}{\partial x}\bigg) - \bigg(f(x,u(x))\cdot
\frac{\partial u}{\partial x}\bigg) + \int_{u(x)}^{v(x)}
\frac{\partial f}{\partial x}~~ dy
\end{align}
For example, we have
\begin{align}
\ptl{}{x_j} \bigg( \int_{x_1}^{x_j} t^k \prod_{i=1}^n (t - x_i)^m  dt \bigg) &= x_j^k \prod_{i=1}^n (x_j - x_i)^m - 0 + \int_{x_1}^{x_j}  \ptl{}{x_j} \bigg( t^k \prod_{i=1}^n (t - x_i)^m \bigg) \\
&= 0-0+\int_{x_1}^{x_j} (-m) t^k (t-x_j)^{m-1}
\prod_{\stackrel{i=1}{i \neq j}}^n (t - x_i)^m dt. \label{xjk1}
\end{align}
A similar calculation of $\ptl{}{x_l}$ for the cases $l=1$ and $l
\neq 1,j$ gives the more general rule
\begin{align}
\ptl{}{x_l} \bigg( \int_{x_1}^{x_j} t^k \prod_{i=1}^n (t - x_i)^m
dt \bigg) &= \int_{x_1}^{x_j} (-m) t^k (t-x_l)^{m-1}
\prod_{\stackrel{i=1}{i \neq l}}^n (t - x_i)^m dt. \label{dqt}
\end{align}
Repeating this differentiation gives
\begin{align}
\ptl{p}{x_l} \bigg( \int_{x_1}^{x_j} t^k \prod_{i=1}^n (t - x_i)^m
dt \bigg) &= \int_{x_1}^{x_j} (-1)^p (m)_p t^k (t-x_l)^{m-p}
\prod_{\stackrel{i=1}{i \neq l}}^n (t - x_i)^m dt \label{dpqt}
\end{align}
for $p \leq m$. Expanding $Q_T^{k,m}$ according to the definition
gives
\begin{align} \label{biglimit}
\lim_{x_j \rightarrow x_1} \frac{Q_T^{k,m}}{(x_j - x_1)^{2m + 1}} =
\lim_{x_j \rightarrow x_1}  \frac{\int_{x_1}^{x_j} t^k \prod_{i=1}^n
(t-x_i)^m  dt}{(x_j-x_1)^{2m+1}}
\end{align}
which is an indeterminate expression of the form $\frac{0}{0}$.
Applying L'Hopital's rule and evaluating the numerator with
(\ref{xjk1}) gives that expression (\ref{biglimit}) equals
\begin{align*}
& \lim_{x_j \rightarrow x_1} \frac{ \ptl{}{x_j} \bigg( \int_{x_1}^{x_j} t^k \prod_{i=1}^n (t - x_i)^m  dt \bigg ) }{ \ptl{}{x_j} (x_j - x_1)^{2m + 1}} \\
=& \lim_{x_j \rightarrow x_1}  \frac{\int_{x_1}^{x_j} (-m) t^k
(t-x_j)^{m-1} \prod_{\stackrel{i=1}{i \neq j}}^n (t - x_i)^m dt }{
(2m+1)(x_j-x_1)^{2m}}
\end{align*}
which is still indeterminate.  However, after $m$ applications of
L'Hopital's rule we obtain
\begin{align*}
\lim_{x_j \rightarrow x_1}  \frac{ (-1)^m\cdot m! \int_{x_1}^{x_j}
t^k\prod_{\stackrel{i=1}{i \neq j}}^n (t - x_i)^m dt }{
(2m+1)(2m)(2m-1) \cdots (m+2) (x_j - x_1)^{m+1}}
\end{align*}
and one more application of L'Hopital's rule, evaluated this time
with the Fundamental Theorem of Calculus, yields
\begin{align}
\lim_{x_j \rightarrow x_1}  \frac{ (-1)^m\cdot m! \cdot x_j^k
\prod_{\stackrel{i=1}{i \not = j}}^n (x_j-x_i)^m }{ (2m+1)(2m)(2m-1)
\cdots (m+1) (x_j - x_1)^m}.
\end{align}
We now cancel the term $(x_j-x_1)^m$ from both numerator and
denominator and the Proposition is proven.
\end{proof}

The polynomiality of $\limj \frac{\Q{k}{m}}{(x_j - x_1)^{2m + 1}}$
immediately gives that ${(x_j - x_1)^{2m + 1}} \big| \Q{k}{m}$.

\begin{Prop} \label{quasi}
The polynomial $\QI \in \gamma_T \qss$ for all $k, m$.
\end{Prop}

\begin{proof}
By Theorem \ref{Thm:Main} we have that $\gamma_T \qss = \gamma_T{\bf
R} \cap V_T^{2m+1}{\bf R}$.  Hence the result is proved by the
previous two propositions.
\end{proof}

We now show that the polynomials $Q_T^{k,m}$ form a basis for the
hook shape $[n-1,1]$.
For this proof, we use Felder and Veselov's Hilbert series result,
as stated in Equation (\ref{eq:Hilb}).
Furthermore, they show in \cite{FV2} that $\qss^*$ affords the
left-regular representation, so that one can break up a basis for
$\qss^*$ into a set of bases for the various isotypic components. In
particular, this shows for $T$ of shape $[n-1,1]$ that the
projection of the quotient $\gamma_T \qss^*$ has Hilbert series
given by
$$\sum_{k=0}^{n-2} q^{mn+1+k}.$$  With this result in mind, we now
prove the following main theorem.

\begin{Thm} \label{Thm:Basis}
The set $$\{Q_T^{0,m}, Q_T^{1,m}, Q_T^{2,m}, \dots, Q_T^{n-2,m}\}
$$ is a basis for $\gamma_T \qss^*$.  
\end{Thm}

\begin{proof}
We first note that $Q_T^{k,m}$ has degree $mn+k+1$, and in
particular, each of these elements are of different degrees, and
matching that of the Hilbert series. Since the set $S = \{Q_T^{0,m},
Q_T^{1,m}, Q_T^{2,m}, \dots, Q_T^{n-2,m}\}$ has size $n-1$, proving
$S$ is linearly independent in $\gamma_T \qss^*$ shows that $S$ is a
basis for $\gamma_T \qss^*$.

Since the quotient $\qss^*$ is graded and the polynomials $Q_T^{k,m}$ are of
different degrees as $k$ varies, it suffices to show that
$Q_T^{k,m}$ is nonzero in the quotient for $0 \le k \le n-2$.  Put
another way, we must show that $Q_T^{k,m}$ is not in the ideal of
$\gamma_T \qss$ generated by $\langle e_1, \dots, e_n \rangle$.
Equivalently we must show that polynomials of the form
$$ P_k = Q_T^{k,m} + A_1Q_T^{k-1,m} + \dots + A_{k-1}Q_T^{1,m} + A_k Q_T^{0,m}$$
(where the $A_i$ are symmetric functions of degree $i$) can only equal
$0$ if $k \ge n-1$.

In fact, we use the explicit formulas  for
$\lim_{x_j\rightarrow x_1}Q_T^{k,m}/V_T^{2m+1}$
given by Proposition~\ref{xjkfactor} to show the stronger statement
 $$\lim_{x_j\rightarrow x_1} P_k/V_T^{2m+1} = 0 \implies k \ge n-1
 $$
regardless of the choice of the symmetric functions.  Letting
$\widetilde{A_i}$ denote the limit $x_j\rightarrow x_1$ applied to the
symmetric function $A_i$, and assuming w.l.o.g. that $j=2$, we have
\begin{align*}
\lim_{x_2\rightarrow x_1} P_k/V_T^{2m+1} &= 0 \\
\implies \bigg({(-1)^m m!^2 \over (2m+1)!} \prod_{i=3}^n
(x_1-x_i)^m \bigg)\bigg( x_1^k + \widetilde{A_1}x_1^{k-1} + \dots +
\widetilde{A_{k-1}}x_1 + \widetilde{A_k}\bigg) &= 0 \\
\implies x_1^k + \widetilde{A_1}x_1^{k-1} + \dots +
\widetilde{A_{k-1}}x_1 + \widetilde{A_k} &= 0\\
\implies \lim_{x_2 \to x_1} \left( x_1^k + A_1 x_1^{k-1} + \dots + A_k
\right) &= 0
\end{align*}
Setting
\begin{align*}
Q(x_1,\dots, x_n) &= x_1^k + A_1x_1^{k-1} + \dots + A_k \\
\intertext{we must have}
Q(x_1,\dots, x_n) &= (x_2-x_1)\cdot R(x_1,\dots, x_n)
\end{align*}
However, $Q$ must be symmetric with respect to all pairs of variables
not involving $x_1$.  Thus, for any $\sigma \in S_{\{2,3,\dots, n\}}$, 
$\sigma Q = Q$
and so $$Q(x_1,\dots, x_n) = \sigma Q(x_1,\dots, x_n) =
(x_{\sigma(2)}-x_1)\cdot \sigma R(x_1,\dots, x_n),$$ and so
$\prod_{i=2}^n (x_i-x_1)$ divides $Q(x_1,\dots, x_n)$. Consequently,
$k$, which is the degree of $Q(x_1,\dots, x_n)$, must be greater
than or equal to $n-1$.
\end{proof}

\section{A More Explicit Description of $Q_T^{k,m}$} \label{Sec:EvalIntegrals}

We now know that the set $\{Q_T^{k,m}\}_{k=0}^{n-2}$ is indeed a basis
for $\gamma_T \qss^*$.  In this section we show an even more explicit
formula for the $Q_T^{k,m}$'s.   Throughout, we shall assume without
loss that the element in the second row of $T$ (which we have been
calling $j$) is $2$. Since $V_T = (x_2-x_1)$ divides $\QI$, we change
variables to understand $\QI$ from a more combinatorial point of view.
Namely, we expand with respect to the variables
$$ Z = \{x_1,\widehat{x_2},\dots, x_n, z = x_2-x_1\}.$$
This is in contrast to the usual set of variables
$$ X = \{x_1,x_2,\dots, x_n\} .$$

\begin{Thm} \label{newdecomp}
  The coefficient of $(x_2-x_1)^r = z^r$ in $\QI$ (when expanded with respect
  to $Z$) is
  \begin{eqnarray*}
        &{m! \over r(r-1)(r-2)\cdots (r-m)}& \sum_{i_3=0}^m \sum_{i_4=0}^m
        \cdots \sum_{i_n=0}^m (-1)^{m+i_3+i_4+\dots + i_n} 
        {m \choose i_3}{m\choose i_4}\cdots {m\choose i_n}\times \\
        && \hspace{-6em} {k+m(n-2)-i_3-i_4-\dots -i_n \choose r-(2m+1)}
        x_1^{(k+m(n-2)-i_3-i_4-\dots-i_n)-(r-(2m+1))}x_3^{i_3}x_4^{i_4}\cdots
        x_n^{i_n}.
\end{eqnarray*}
\end{Thm}

\begin{proof}
  We begin by evaluating the integral 
  \begin{align} \label{2ndint}
        \int_{t_m=x_1}^{x_2} \int_{t_{m-1}=x_1}^{t_m}\cdots
        \int_{t_0=x_1}^{t_1} (-1)^m m!~t_0^k 
        \prod_{\stackrel{i=1}{i \neq 2}}^n (t_0-x_i)^m~ dt_0 \cdots dt_m.
  \end{align} 
  We will then show that this integral is another way of writing
  $Q_T^{k,m}$.  We begin our evaluation of (\ref{2ndint}) by expanding
  each of the $(t_0-x_i)^m$ for $i\geq 3$ by the binomial theorem,
  thus obtaining
  \begin{align*}
        \int_{t_m=x_1}^{x_2} &\int_{t_{m-1}=x_1}^{t_m} \cdots \int_{t_0=x_1}^{t_1} 
        (-1)^m m!~t_0^k(t_0-x_1)^m \bigg( \sum_{i_3=0}^m (-1)^{i_3} {m\choose i_3}
        t_0^{m-i_3} x_3^{i_3}\bigg) \times \\ 
        &\bigg( \sum_{i_4=0}^m (-1)^{i_4} {m\choose i_4} t_0^{m-i_4} x_4^{i_4}\bigg)
        \cdots \bigg( \sum_{i_n=0}^m (-1)^{i_n} {m\choose i_n} t_0^{m-i_n}
        x_n^{i_n}\bigg) ~ dt_0 \cdots dt_m.
  \end{align*}
  This quantity simplifies to
  \begin{align*}
        \int_{t_m=x_1}^{x_2} & \int_{t_{m-1}=x_1}^{t_m}\cdots 
        \int_{t_0=x_1}^{t_1} m!~ t_0^k (t_0-x_1)^m 
        \bigg(\sum_{i_3=0}^m\sum_{i_4=0}^m\cdots \sum_{i_n=0}^m 
        (-1)^{m+i_3+i_4+\dots + i_n}\times \\ 
        &{m\choose i_3}{m\choose i_4}\cdots {m\choose i_n}x_3^{i_3}x_4^{i_4}
        \cdots x_n^{i_n}~~\cdot~~t_0^{m\cdot (n-2) - i_3 - i_4-\dots -i_n}\bigg)
        ~dt_0 \cdots dt_m,
  \end{align*}
  and by rearranging we obtain
  \begin{align*}
        m!&~\sum_{i_3=0}^m\sum_{i_4=0}^m\cdots \sum_{i_n=0}^m
        (-1)^{m+i_3+i_4+\dots + i_n}{m\choose i_3}{m\choose i_4}\cdots
        {m\choose i_n} x_3^{i_3}x_4^{i_4}\cdots x_n^{i_n} \times \\
        & \bigg(\int_{t_m=x_1}^{x_2} \int_{t_{m-1}=x_1}^{t_m}\cdots
        \int_{t_0=x_1}^{t_1} t_0^{k +m(n-2) - i_3 - i_4-\dots -i_n} (t_0-x_1)^m
        ~dt_0 \cdots dt_m\bigg).
  \end{align*}
  For convenience of notation we let $K = k +m\cdot (n-2) - i_3 -
  i_4-\dots -i_n$, allowing us to write the above as
  \begin{align*}
        m!&~\sum_{i_3=0}^m\sum_{i_4=0}^m\cdots \sum_{i_n=0}^m
        (-1)^{m+i_3+i_4+\dots + i_n}{m\choose i_3}{m\choose i_4}\cdots
        {m\choose i_n} x_3^{i_3}x_4^{i_4}\cdots x_n^{i_n} \times \\ 
        &\bigg(\int_{t_m=x_1}^{x_2} \int_{t_{m-1}=x_1}^{t_m}\cdots 
        \int_{t_0=x_1}^{t_1} t_0^{K}(t_0-x_1)^m ~dt_0 \cdots dt_m\bigg).
  \end{align*}
  At this point, we rewrite $t_0^K$ as $(x_1+(t_0-x_1))^K$, which
  allows us to simplify $t_0^K (t_0-x_1)^m$ as  $\sum_{R=0}^{K} {K
  \choose R} x_1^{K-R} (t_0-x_1)^{R+m},$ hence we conclude
  (\ref{2ndint}) equals
  \begin{align*}
        m!&~\sum_{i_3=0}^m\sum_{i_4=0}^m\cdots \sum_{i_n=0}^m
        (-1)^{m+i_3+i_4+\dots + i_n}{m\choose i_3}{m\choose i_4}\cdots
        {m\choose i_n} \times \\
        &\sum_{R=0}^K x_1^{K-R}\cdot x_3^{i_3}x_4^{i_4}\cdots x_n^{i_n}
        \bigg(\int_{t_m=x_1}^{x_2} \int_{t_{m-1}=x_1}^{t_m}\cdots 
        \int_{t_0=x_1}^{t_1}(t_0-x_1)^{m+R} ~dt_0 \cdots dt_m\bigg),
  \end{align*}
  and the inside integral is easily seen to evaluate to
  $${(x_2-x_1)^{2m+1+R}\over (R+2m+1)(R+2m)\cdots (R+m+1)}.$$
  Finally, we let $r = R +(2m+1)$, i.e. $R = r-(2m+1)$, so that $r$
  signifies the power of $z=(x_2-x_1)$ in the expression. Thus the
  coefficient of $z^r$ is as claimed in the statement of the theorem.
  
  It remains to show that $Q_T^{k,m}$ is in fact equal to the quantity
  in (\ref{2ndint}).  We note that the argument above shows that
  $z^{2m}$ divides (\ref{2ndint}).  We also know from
  Proposition~\ref{xjkfactor} that $z^{2m}$ divides $Q_T^{k,m}$.
  Thus, showing
  $$ \frac{\partial^m}{\partial z^m} Q_T^{k,m} =
  \frac{\partial^m}{\partial z^m} (\ref{2ndint}) $$
  shows equality of $Q_T^{k,m}$ and (\ref{2ndint}). Furthermore, the
  operator $\frac{\partial}{\partial z}$ applied to a polynomial in
  the generating set $Z$ is equivalent to the operator
  $\frac{\partial}{\partial x_2}$ applied to the same polynomial in
  the generating set $X$.  Thus, we will show $Q_T^{k,m} =
  (\ref{2ndint})$ by showing
  \begin{align} \label{eq:expandgoal}
    \frac{\partial^m}{\partial x_2^m} Q_T^{k,m} =
    \frac{\partial^m}{\partial x_2^m} (\ref{2ndint}) 
  \end{align}
  For the LHS, consider the function $f(t) = t^k \prod_{i=1}^n
  (t-x_i)^m.$  As in the previous section, we use Leibniz's formula to
  obtain
  \begin{align*}
        {\partial \over \partial x_2} \int_{t=x_1}^{x_2} f(t)~ dt
        &= f (x_2) + \int_{t=x_1}^{x_2} \bigg(\ddtw f(t)\bigg) dt \\
        &= \hspace{3.5em} \int_{t=x_1}^{x_2} \bigg(\ddtw f(t)\bigg)
        dt.
  \end{align*}
  After iterating $m$ times, we obtain 
  \begin{align*} 
        \ddtwm \int_{x_1}^{x_2} t^k \prod_{i=1}^n (t-x_i)^m dt = (-1)^m
        m!~\int_{x_1}^{x_2} t^k \prod_{\stackrel{i=1}{i\not = 2}}^n
        (t-x_i)^m dt.
  \end{align*}
  For the RHS, we let 
  $$ g(t,m) = \int_{t_{m-1}=x_1}^{t}\cdots \int_{t_0=x_1}^{t_1} 
  (-1)^m m!~t_0^k \prod_{\stackrel{i=1}{i\not= 2}}^n 
  (t_0-x_i)^m~ dt_0 \cdots dt_{m-1},$$
  and note by the Fundamental Theorem of Calculus that
  \begin{align*}
        {\partial \over \partial x_2} \int_{t_m=x_1}^{x_2} 
        g(t, m) ~dt_m &= \int_{t_{m-1}=x_1}^{x_2} g(t, m-1)~dt_{m-1}
  \end{align*}
  since the integrand does not include variable $x_2$.
  Thus 
  \begin{align*}
        &{\ddtwm}\int_{t_m=x_1}^{x_2} \int_{t_{m-1}=x_1}^{t_m} \cdots
        \int_{t_0=x_1}^{t_1} (-1)^m m!~t_0^k 
        \prod_{\stackrel{i=1}{i\not= 2}}^n (t_0-x_i)^m~ dt_0 \cdots dt_m
        \\
        =& (-1)^m m!~\int_{x_1}^{x_2} t^k
        \prod_{\stackrel{i=1}{i\not = 2}}^n (t-x_i)^m dt \\  
        =& {\ddtwm}\int_{x_1}^{x_2} t^k \prod_{i=1}^n
        (t-x_i)^m dt
  \end{align*}
  which establishes (\ref{eq:expandgoal}).
\end{proof}

\section{The Action of Operator $L_m$} \label{Sec:LmAction}

In this section, we discuss a further property of our basis for
$\gamma_T QI_m^*$ for $T$ a standard Young tableau of shape
$[n-1,1]$. In particular, as discussed in \cite{FV,GW} and elsewhere
in the literature, there is a natural family of operators which act
on the quasiinvariants and are denoted by $L_m$.  In particular
these operators are defined, in the symmetric group case, as $$L_m =
\sum_{i=1}^n \frac{\partial^2}{\partial x_i^2} - 2m \sum_{1 \leq i <
j \leq n} \frac{1}{x_i - x_j} \left( \ptl{}{x_i} - \ptl{}{x_j}
\right).$$ The action of $L_m$ on our basis is striking.  In
particular, we obtain the following formulas for this action:
\begin{Thm} \label{Thm:Lm} $L_m(Q_T^{k,m}) = k(k-1)Q_T^{k-2,m}$ for $k \geq 2$ and equals zero for $k = 0$ or $1$.
\end{Thm}

The significance of this formula is how $L_m$ acts as second
differentiation with respect to the basis $\{Q_T^{0,m},Q_T^{1,m},
\dots, Q_T^{n-2,m}\}$.  This action naturally generalizes the action
of $L_0 = \sum_{i=1}^n \frac{\partial^2}{\partial x_i^2}$ on the polynomial ring
$QI_0$.

\begin{proof}
We now proceed with the proof of Theorem \ref{Thm:Lm}.  For $m = 0$,
we have $Q_T^{k,0} = \frac{x_j^{k+1} - x_1^{k+1}}{k + 1}$ by
(\ref{qtk0}). Thus
\begin{align*}
L_0 Q_T^{k,0} &= \left( \ptl{2}{x_1} + \ptl{2}{x_j} \right) \left( \frac{x_j^{k+1} - x_1^{k+1}}{k + 1} \right) - 0 \\
&= (k) (x_j^{k-1} - x_1^{k-1}) \\
&= (k)(k-1) Q_T^{k-2,0}.
\end{align*}
We state here some useful identities which are valid for $m \ge 1$.
First, we have
\begin{align}
& \int_{x_1}^{x_j} \ptl{2}{t} \left( t^k \prod_{i=1}^n (t - x_i)^m \right) dt \label{1int} \\
=& \int_{x_1}^{x_j} \left( \ptl{2}{t} t^k \right) \prod_{i=1}^n (t - x_i)^m dt \label{2int}\\
&+ 2 \int_{x_1}^{x_j} \left( \ptl{}{t} t^k \right) \left( \ptl{}{t} \prod_{i=1}^n (t - x_i)^m \right) dt \label{3int}\\
&+ \int_{x_1}^{x_j} t^k \left( \ptl{2}{t} \prod_{i=1}^n (t - x_i)^m
\right) dt \label{4int}.
\end{align}
We also have
\begin{align}& \int_{x_1}^{x_j} \ptl{}{t} \bigg[ t^k \ptl{}{t} \left( \prod_{i=1}^n (t - x_i)^m \right) \bigg] dt
\label{5int} \\ \nonumber =&\int_{x_1}^{x_j} \left( \ptl{}{t} t^k
\right) \left( \ptl{}{t} \prod_{i=1}^n (t - x_i)^m \right) dt +
\int_{x_1}^{x_j} t^k \left( \ptl{2}{t} \prod_{i=1}^n (t - x_i)^m
\right) dt \\ \nonumber =& \frac{1}{2}(\ref{3int}) + (\ref{4int}).
\end{align}
Additionally, we can compute (\ref{2int}) as follows:
\begin{align*}
\int_{x_1}^{x_j} \left( \ptl{2}{t} t^k \right) \prod_{i=1}^n (t -
x_i)^m dt
&= k(k - 1) \int_{x_1}^{x_j} t^{k-2} \prod_{i=1}^n (t - x_i)^m dt   \\
&= k(k - 1) Q_T^{k-2,m}.
\end{align*}
Now, for $m \ge 2$ we recall equations (\ref{dqt}) and (\ref{dpqt})
where we used Leibniz's integral formula to obtain
\begin{align*}
\ptl{}{x_i} (Q_T^{k,m}) &= \int_{x_1}^{x_j}  (-m) t^k (t -
x_i)^{m-1} \prod_{\stackrel{l=1}{l \neq i}}^n (t - x_l)^m dt
\end{align*}
and
\begin{align*}
\ptl{2}{x_i}(Q_T^{k,m}) &= \int_{x_1}^{x_j}  m(m-1) t^k (t -
x_i)^{m-2} \prod_{\stackrel{l=1}{l \neq i}}^n(t - x_l)^m dt.
\end{align*}
Using these we can compute
\begin{align*}
&\sum_{1 \le i < l \le n} \frac{1}{x_i - x_l} \left( \ptl{}{x_i} - \ptl{}{x_l} \right) Q_T^{k,m}\\
=& (-m) \int_{x_1}^{x_j} t^k \sum_{1 \le i < l \le n} \frac{1}{x_i - x_l} \Big( \bigg[ (t - x_i)^{m-1} \prod_{\stackrel{p=1}{p \neq i}}^n (t - x_p)^m \bigg] - \bigg[ (t - x_l)^{m-1} \prod_{\stackrel{p=1}{p \neq l}}^{n} (t - x_p)^m \bigg] \Big) dt \\
=& (-m) \int_{x_1}^{x_j} t^k \sum_{1 \le i < l \le n} \frac{1}{x_i - x_l} \bigg[ (t - x_i)^{m-1} (t-x_l)^{m} - (t - x_i)^m (t - x_l)^{m-1} \bigg] \prod_{\stackrel{p=1}{p \neq i,l}}^{n} (t - x_p)^m dt \\
=& (-m) \int_{x_1}^{x_j} t^k \left[ \prod_{p=1}^{n} (t - x_p)^{m-1} \right] \sum_{1 \le i < l \le n} \bigg[ \frac{(t - x_l) - (t - x_i)}{x_i - x_l} \prod_{\stackrel{p=1}{p \neq i,l}}^{n} (t - x_p) \bigg] dt \\
=& (-m) \int_{x_1}^{x_j} t^k \sum_{1 \leq i < l \leq n} (t -
x_i)^{m-1}(t - x_l)^{m-1} \prod_{\stackrel{p=1}{l \neq i,l}}^{n} (t
- x_l)^m dt
\end{align*}
and hence
\begin{align*}
L_m (Q_T^{k,m}) &= m (m-1) \int_{x_1}^{x_j} t^k \sum_{i=1}^n (t-x_i)^{m-2} \prod_{\stackrel{l=1}{l \neq i}}^n (t - x_l)^m dt \\
&+ 2m^2 \int_{x_1}^{x_j} t^k \sum_{1 \leq i < l \leq n} (t -
x_i)^{m-1}(t - x_l)^{m-1} \prod_{\stackrel{p=1}{p \not = i,j}}^n (t
- x_l)^m dt.
\end{align*}
We recognize this expression as being nothing more than
\begin{align} \label{firstint}
L_m (Q_T^{k,m}) &= \int_{x_1}^{x_j} t^k \bigg( \ptl{2}{t}
\prod_{i=1}^n (t - x_i)^m \bigg) dt.
\end{align}
Now, if we evaluate (\ref{1int}) by the fundamental theorem of
calculus we get
\begin{align*}
&\int_{x_1}^{x_j} \ptl{2}{t} \left( t^k \prod_{i=1}^n (t - x_i)^m \right) dt \\
&= k t^{k-1} \prod_{i=1}^n (t - x_i)^m + t^k \sum_{i=1}^n (t - x_i)^{m-1} \prod_{\stackrel{l=1}{l \neq i}} (t - x_l)^m \bigg|_{t=x_1}^{t=x_j} \\
&= 0.
\end{align*}
Thus we have $(\ref{2int}) + (\ref{3int}) + (\ref{4int}) = 0$.
Similarly we can evaluate (\ref{5int}) to obtain
\begin{align*}
& \frac{1}{2}(\ref{3int}) + (\ref{4int}) \\
=& \int_{x_1}^{x_j} \ptl{}{t} \bigg[ t^k \ptl{}{t} \left( \prod_{i=1}^n (t - x_i)^m \right) \bigg] dt \\
=& t^k \sum_{i=1}^n (t - x_i)^{m-1} \prod_{\stackrel{l=1}{l \neq i}} (t - x_l)^m \bigg|_{t=x_1}^{t=x_j} \\
=& 0.
\end{align*}
Using $(\ref{2int}) + (\ref{3int}) + (\ref{4int}) = 0$ and
$\frac{1}{2}(\ref{3int}) + (\ref{4int}) = 0$, we obtain
$(\ref{2int}) = (\ref{4int})$. So by (\ref{firstint}) we have
\begin{align*}
L_m(Q_T^{k,m}) &= (\ref{4int}) \\
&= (\ref{2int}) \\
&= k(k - 1) Q_T^{k-2,m}.
\end{align*}
Thus we have proven the theorem for $m = 0$ and $m \geq 2$.  For $m
= 1$ similar logic works. We first compute
\begin{align*}
\sum_{i=1}^n \ptl{2}{x_i} \left( \int_{x_1}^{x_j} t^k \prod_{l=1}^n (t - x_l) dt \right) &= \sum_{i=1}^n \ptl{}{x_i} \left( - \int_{x_1}^{x_j} t^k \prod_{\stackrel{l=1}{l \neq i}}^n (t - x_l) dt \right) \\
&= - \left( \ptl{}{x_1} + \ptl{}{x_j} \right) \left( \int_{x_1}^{x_j} t^k \prod_{\stackrel{l=1}{l \neq i}}^n (t - x_l) dt \right) \\
&= x_1^k \left( \prod_{i=2}^n (x_1 - x_i) \right) -  x_j^k \left(
\prod_{\stackrel{i=1}{i \neq j}}^n (x_j -x_i) \right)
\end{align*}
and we can easily verify that this quantity is also equal to
\begin{align*}
- \int_{x_1}^{x_j} \ptl{2}{t} \left( t^k \prod_{i=1}^n (t - x_i)
\right) dt = -(\ref{1int}).
\end{align*}
With that in hand, we also compute
\begin{align*}
&\sum_{1 \le i < l \le n} \frac{1}{x_i - x_l} \left( \ptl{}{x_i} - \ptl{}{x_l} \right) \int_{x_1}^{x_j} t^k \prod_{l=1}^n (t - x_l) dt \\
=& \sum_{1 \le i < l \le n} \frac{1}{x_i - x_l} \int_{x_1}^{x_j} t^k \left( \prod_{\stackrel{p=1}{p \neq i}}^n (t - x_p) - \prod_{\stackrel{p=1}{p \neq l}}^n (t - x_p) \right) dt \\
=& \sum_{1 \le i < l \le n} \int_{x_1}^{x_j} t^k \left( \prod_{\stackrel{p=1}{p \neq i,l}}^n (t - x_p) \big[ \frac{(t - x_l) - (t - x_i)}{x_i - x_l} \big] \right) dt \\
=& \sum_{1 \le i < l \le n} \int_{x_1}^{x_j} t^k
\prod_{\stackrel{p=1}{p \neq i,l}}^n (t - x_p) dt
\end{align*}
Combining this with the following:
\begin{align*}
(\ref{4int}) &= \int_{x_1}^{x_j}  t^k \left( \ptl{2}{t} \prod_{i=1}^n (t - x_i) \right) dt \\
&= 2 \sum_{1 \le i < l \le n} \int_{x_1}^{x_j} t^k
\prod_{\stackrel{p=1}{p \neq i,l}}^n (t - x_p) dt
\end{align*}
shows that we have $L_1 Q_T^{k,1} = -(\ref{1int}) + (\ref{4int})$.
Further, we have
\begin{align*}
(\ref{5int}) &= \int_{x_1}^{x_j} \ptl{}{t} \bigg[ t^k \ptl{}{t} \left( \prod_{i=1}^n (t - x_i) \right) \bigg] dt \\
&= \int_{x_1}^{x_j} \ptl{}{t} \bigg[ t^k \sum_{i=1}^n  \prod_{\stackrel{l=1}{l \neq i}}^n (t - x_l) \bigg] dt \\
&= x_j^k \left( \prod_{\stackrel{i=1}{i \neq j}}^n (x_j -x_i) \right) - x_1^k \left( \prod_{i=2}^n (x_1 - x_i) \right) \\
&= (\ref{1int}).
\end{align*}
Hence we conclude
\begin{align*}
(\ref{1int}) &=(\ref{2int})+(\ref{3int})+(\ref{4int}) \\
&= (\ref{2int}) + ( -2(\ref{4int}) + 2(\ref{5int}) ) + (\ref{4int}) \\
&= (\ref{2int}) - (\ref{4int}) + 2(\ref{1int} ) \\
\Rightarrow (\ref{2int}) &= -(\ref{1int}) + (\ref{4int}) \\
&= L_1 Q_T^{k,1}
\end{align*} thus completing the proof.
\end{proof}

\section{Change of Basis Matrix for Quasiinvariants}
\label{Sec:ChangeOfBasis}

We now turn our attention to analyzing the relationship between the
$m$-quasiinvariants and the $(m+1)$-quasiinvariants.  In particular, recall
that ${\bf QI}_m \supset {\bf QI}_{m+1} \supset \Lambda_n$ for all
$m$, and so we can expand any basis for ${\bf QI}_{m+1}$ in terms
of a basis for ${\bf QI}_{m}$ over the ring $\Lambda_n$ of symmetric
functions.
Each of these bases has $n!$ elements, and thus we obtain a square
change of basis matrix. Since the only invertible symmetric
functions are the constants, any choice of bases for ${\bf QI}_{m}$
and ${\bf QI}_{m+1}$ will yield a change of basis matrix with the
same determinant up to a scalar multiple.  We in fact obtain the
following explicit formula for these determinants:

\begin{Thm} \label{Thm:Park} For all $n$ and $m$, any matrix
  expressing the expansion of a basis for ${\bf QI}_{m+1}$ in terms of
  a basis for $\qss$ with symmetric function coefficients will have a
  determinant equal to a scalar multiple of $(\Delta_n)^{n!}$, where $\Delta_n$
  denotes the Vandermonde determinant $\prod_{1 \leq i < j \leq n}
  (x_i-x_j).$
\end{Thm}

\begin{Lem} \label{Lem:subring}
The ring $\Delta_n^2\cdot {\bf QI}_m$ is a subring of ${\bf
QI}_{m+1}$.
\end{Lem}

\begin{proof}
Since $\Delta_n$ is antisymmetric, $\Delta_n^2$ is a symmetric
function and by Proposition \ref{fpq}, we have for polynomial $P \in
{\bf QI}_m$, 
$$(1-(i,j))(\Delta_n^2 P) = \Delta_n^2 ((1-(i,j))P) =
\Delta_n^2(x_i-x_j)^{2m+1}P^\prime$$
for all $1 \leq i < j \leq n.$ In particular, for all $1 \leq i < j
\leq n$, the polynomial $(1-(i,j))(\Delta_n^2 P)$ is divisible by
$(x_i-x_j)^{2m+3}$ and thus $\Delta_n^2P$ is $(m+1)$-quasiinvariant.
\end{proof}

\begin{proof} [Proof of Theorem \ref{Thm:Park}.]
  We begin picking a basis (over $\Lambda_n$) of homogeneous
  polynomials $\{ \beta_{S,T} \}$ for $\qss$ where $S$ and $T$ range
  over all pairs of standard tableaux of the same shape and the degree
  of $\beta_{S,T}$ is $m \left( \binom{n}{2} -
  content(\lambda(T)) \right) + cocharge(T)$. We know this is possible
  by the Hilbert series stated in (\ref{eq:Hilb}).  We similarly pick a
  basis $\{ \alpha_{S,T} \}$ for ${\bf QI}_{m+1}$. Now, by
  Lemma~\ref{Lem:subring} we have the following containments:
  $$ \Delta_n^2\cdot{\bf QI}_m  \subset {\bf QI}_{m+1}  \subset {\bf QI}_m.$$
  We label these modules $M_1, M_2,$ and $M_3$ respectively and use
  the basis $\{ \Delta_n^2 \beta_{S,T} \}$ for $M_1$.  We set $A$ to
  be the change of basis matrix between $M_1$ and $M_2$ and $B$ to be
  the matrix from $M_2$ to $M_3$.  We immediately obtain that $AB =
  diag(\Delta_n^2)$.  Thus, in particular, $\det(B)$ divides
  $\Delta_n^{2(n!)}$.

  We now consider the degree of an arbitrary non-zero term of
  $\det(B)$. By considering the difference in degrees of all basis
  elements, we must have
  \begin{align*}
        degree(\det(B)) & \\
        &=\left( \sum_{T \in ST(n)}^{} f_{\lambda(T)} (m+1) \left(
        \binom{n}{2} - content(\lambda(T)) \right) + cocharge(T) \right)
        \\ &~~ -
        \left( \sum_{T' \in ST(n)} f_{\lambda(T')}m \left( 
        \binom{n}{2} - content(\lambda(T')) \right) + cocharge(T') \right)
        \\
        &= \sum_{\lambda \vdash n}^{} f_\lambda^2 \left(
        \binom{n}{2} - content(\lambda) \right)
  \end{align*}
  However, it is easy to see that $f_\lambda = f_{\lambda'}$ and
  $content(\lambda) = -content(\lambda')$, where $\lambda'$ is the conjugate 
(or transpose) of partition $\lambda$. Hence we have
  $$ degree(\det(B)) = \sum_{\lambda \vdash n}^{} f_\lambda^2
  \binom{n}{2} = \binom{n}{2} n!. $$
  Since $\det(B)$ is a symmetric function of degree $\binom{n}{2} n!$
  which divides $\Delta_n^{2(n!)}$, and $\Delta_n^2$ has no
  nontrivial symmetric function factors, we conclude that $\det(B)$
  equals $\Delta_n^{n!}$, up to a scalar multiple.
\end{proof}

\section{Conclusions and Open Problems} \label{Sec:Conclusions}

In this paper, we provided a decomposition of the ring of
$m$-quasiinvariants into isotypic components and gave two easy
criteria for characterizing such elements.  One application of this
new characterization was an explicit description of a basis for the
isotypic component corresponding to shape $[n-1,1]$.  In particular
such basis elements can be written as either integrals or
algebraically using polynomials with coefficients given as products
of binomial coefficients.

One natural extension of this research involves further analysis of
the representation theoretic aspects of $m$-quasiinvariants.  In
particular can one re-characterize quasiinvariants for other Coxeter
groups using analogous criteria.  Another direction is the
computation of explicit bases for more isotypic components.  It
would be even better if the operator $L_m$ respected these bases in
a similar manner.

\vspace{2em}

\noindent {\bf Acknowledgements.}

The authors are grateful for Adriano Garsia's guidance in this
project as well as the support of the NSF.  We also would like to
thank Vic Reiner for conversations which helped motivate the
analysis of section $8$.

\end{document}